\documentclass[11pt,a4paper,reqno]{amsart}
\usepackage{calc,amsfonts,amsthm,amscd,epsfig,psfrag,amsmath,amssymb,enumerate,graphicx}

\reversemarginpar
\newlength\fullwidth
\numberwithin{equation}{section}

\DeclareMathSymbol{\leqslant}{\mathalpha}{AMSa}{"36} 
\DeclareMathSymbol{\geqslant}{\mathalpha}{AMSa}{"3E} 
\DeclareMathSymbol{\eset}{\mathalpha}{AMSb}{"3F}     
\renewcommand{\leq}{\;\leqslant\;}                   
\renewcommand{\geq}{\;\geqslant\;}                   

\def\1{\ifmmode {1\hskip -3pt \rm{I}} \else {\hbox {$1\hskip -3pt \rm{I}$}}\fi}

\newtheorem{Theorem}{Theorem}[section]
\newtheorem{Lemma}[Theorem]{Lemma}
\newtheorem{Proposition}[Theorem]{Proposition}

\newtheorem{remark}[Theorem]{Remark}

\newtheorem{claim}[Theorem]{Claim}
\newtheorem{definition}[Theorem]{Definition}

\newcommand{\cA}{\ensuremath{\mathcal A}}

\newcommand{\cC}{\ensuremath{\mathcal C}}
\newcommand{\cD}{\ensuremath{\mathcal D}}
\newcommand{\cE}{\ensuremath{\mathcal E}}
\newcommand{\cF}{\ensuremath{\mathcal F}}
\newcommand{\cG}{\ensuremath{\mathcal G}}

\newcommand{\cK}{\ensuremath{\mathcal K}}
\newcommand{\cL}{\ensuremath{\mathcal L}}

\newcommand{\cO}{\ensuremath{\mathcal O}}
\newcommand{\cP}{\ensuremath{\mathcal P}}

\newcommand{\cT}{\ensuremath{\mathcal T}}

\newcommand{\bbE}{{\ensuremath{\mathbb E}} }

\newcommand{\bbN}{{\ensuremath{\mathbb N}} }

\newcommand{\bbP}{{\ensuremath{\mathbb P}} }

\newcommand{\bbR}{{\ensuremath{\mathbb R}} }

\newcommand{\bbZ}{{\ensuremath{\mathbb Z}} }

    \let\d=\delta  
  \let\h=\eta      \let\l=\lambda
      \let\o=\omega      
  \let\s=\sigma    
  
     \let\L=\Lambda 
\let\O=\Omega      

\def\bigno{\bigskip\noindent}
\def\\{\hfill\break}

\def\thsp{\thinspace}

\def\tthsp{\kern .083333 em}

\def\?{\mskip -10mu}
\def\indbox#1{\hbox to \parindent{\hfil\ #1\hfil} }

\newfam\msafam
\newfam\msbfam
\newfam\eufmfam
\def\hexnumber#1{%
  \ifcase#1 0\or 1\or 2\or 3\or 4\or 5\or 6\or 7\or 8\or
  9\or A\or B\or C\or D\or E\or F\fi}
\font\tenmsa=msam10
\font\sevenmsa=msam7
\font\fivemsa=msam5
\textfont\msafam=\tenmsa
\scriptfont\msafam=\sevenmsa
\scriptscriptfont\msafam=\fivemsa
\edef\msafamhexnumber{\hexnumber\msafam}%
\mathchardef\restriction"1\msafamhexnumber16
\mathchardef\ssim"0218
\mathchardef\square"0\msafamhexnumber03
\mathchardef\eqd"3\msafamhexnumber2C
\def\QED{\ifhmode\unskip\nobreak\fi\quad
  \ifmmode\square\else$\square$\fi}
\font\tenmsb=msbm10
\font\sevenmsb=msbm7
\font\fivemsb=msbm5
\textfont\msbfam=\tenmsb
\scriptfont\msbfam=\sevenmsb
\scriptscriptfont\msbfam=\fivemsb

\font\teneufm=eufm10
\font\seveneufm=eufm7
\font\fiveeufm=eufm5
\textfont\eufmfam=\teneufm
\scriptfont\eufmfam=\seveneufm
\scriptscriptfont\eufmfam=\fiveeufm

\def\({\left(}
\def\){\right)}
\let\neper=e
\let\ii=i

\def\ie{\hbox{\it i.e.\ }}
\let\id=\identity

\let\sset=\subset

\def\nep#1{ \neper^{#1}}

\def\tc{\thsp | \thsp}
\let\<=\langle
\let\>=\rangle

\def\Var{ \mathop{\rm Var}\nolimits }

\def\gap{\mathop{\rm gap}\nolimits}

\def\ninf#1{ \| #1 \|_\infty }

\def\inte#1{\lfloor #1 \rfloor}

\outer\def\nproclaim#1 [#2]#3. #4\par{\medbreak \noindent
   \talato(#2){\bf #1 \Thm[#2]#3.\enspace }%
   {\sl #4\par }\ifdim \lastskip <\medskipamount
   \removelastskip \penalty 55\medskip \fi}
\def\thmm[#1]{#1}
\def\teo[#1]{#1}
\def\sttilde#1{%
\dimen2=\fontdimen5\textfont0
\setbox0=\hbox{$\mathchar"7E$}
\setbox1=\hbox{$\scriptstyle #1$}
\dimen0=\wd0
\dimen1=\wd1
\advance\dimen1 by -\dimen0
\divide\dimen1 by 2
\vbox{\offinterlineskip%
   \moveright\dimen1 \box0 \kern - \dimen2\box1}
}
\begin{document}

\title[Facilitated oriented spin models: some non equilibrium
results]{Facilitated oriented spin models:\\ some non equilibrium results}

\author[N. Cancrini]{N. Cancrini}
\email{nicoletta.cancrini@roma1.infn.it}
\address{Dip. Matematica Univ.l'Aquila, 1-67100 L'Aquila, Italy}

\author[F. Martinelli]{F. Martinelli}
\email{martin@mat.uniroma3.it}
\address{Dip.Matematica, Univ. Roma Tre, Largo S.L.Murialdo 00146, Roma, Italy}

\author[R. H. Schonmann]{R.H. Schonmann}
\email{rhs@math.ucla.edu}
\address{Dept. of Mathematics, UCLA, Los Angeles CA90095, USA}
\author[C. Toninelli]{C. Toninelli}
\email{cristina.toninelli@upmc.fr}
\address{Laboratoire de Probabilit\'es et Mod\`eles Al\`eatoires
  CNRS-UMR 7599 Universit\'es Paris VI-VII 4, Place Jussieu F-75252 Paris Cedex 05 France}

\thanks{The work of R.H.S. was partially supported by NSF under grant DMS 0300672. The work of C.T. was partially supported by the French Ministry of Education through the ANR BLAN07-2184264 grant.}

\begin{abstract}
  We begin the rigorous analysis of the relaxation to equilibrium for
  kinetically constrained spin models (KCSM) when the
  initial distribution $\nu$ is different from the reversible one, $\mu$. This setting
 has been intensively studied in the physics literature to 
 analyze the slow dynamics which follows a sudden quench from the liquid to the glass phase.  We concentrate on
  two basic oriented KCSM: the East model on $\bbZ$, for which the
  constraint requires that the East neighbor of 
  the to-be-update vertex is vacant and the model on the binary tree
  introduced in \cite{Aldous:2002p1074}, for which the constraint requires the two
  children to be vacant. It is important to observe that, while the former model is ergodic at any $p\neq 1$, the latter displays an ergodicity breaking
  transition at $p_c=1/2$.
 For the East we prove exponential convergence to
  equilibrium with rate depending on the spectral
 gap if $\nu$ is concentrated on any
  configuration which does not contain a forever blocked site or if
  $\nu$ is a Bernoulli($p'$) product measure for any $p'\neq 1$.  For the model on the binary tree we prove similar results in the regime
  $p,p'<p_c$ and under the (plausible) assumption that the
  spectral gap is positive for $p<p_c$. By constructing a proper test
  function, we also prove that if $p'>p_c$ and $p\leq p_c$ convergence to equilibrium cannot occur for all local functions. Finally, in a short appendix, we present a very simple argument, different
  from the one given in  \cite{Aldous:2002p1074}, based on an elegant combination of some combinatorial results
  together with ``energy barrier'' considerations, which yields the sharp upper bound for
  the spectral gap of East when $p\uparrow 1$.
 \bigno

\end{abstract}

\bigno

\maketitle

\thispagestyle{empty}

\section{Introduction}

Facilitated or kinetically constrained spin models (KCSM) are
interacting particle systems which have been introduced in physics
literature \cite{Fredrickson:1984p1690} \cite{Fredrickson:1985p1686} to model liquid/glass transition and more
generally ``glassy dynamics'' (see  \cite{Ritort:2003p2313}  \cite{Toninelli:2007p2382}). They are defined
on a locally finite, bounded degree, connected graph $\mathcal
G=(V,E)$ with vertex set $V$ and edge set $E$. For most KCSM the
graph is the integer lattice $\bbZ^d$. A configuration is given by
assigning to each site $x\in V$ its occupation variable
$\eta(x)\in\{0,1\}$ which corresponds to an empty or filled site,
respectively.  The evolution is given by a Markovian stochastic dynamics
of Glauber type.  Each site waits an independent, mean one, exponential
time and then, provided the current configuration around it satisfies an
a priori specified constraint, its occupation variable is refreshed to
an occupied or to an empty state with probability $p$ or $1-p$,
respectively.  For each site $x$ the corresponding constraint does not
involve $\h_x$, thus detailed balance w.r.t. Bernoulli($p$) product
measure $\mu$ can be easily verified and the latter is an invariant
reversible measure for the process.

Among the most studied KCSM, we recall the East, North-East and FA-jf models.
The East model    \cite{JACKLE:1991p2315}          is one-dimensional  ($\cG=\bbZ$) and particle
creation/destruction at a given site can occur only if its nearest
neighbor to the right is empty.
The North-East model (N-E)          \cite{REITER:1992p2314}           is instead defined on $\bbZ^2$
and refreshing at $x$ can
occur only if both its North and East  neighbors are empty.
Finally, FA-jf models \cite{Fredrickson:1984p1690} are usually defined on $\mathbb Z^d$ 
and the constraint requires at least $j$ empty sites among the $2d$ nearest
neighbors (with $j$ satisfying $1\leq j\leq d$).
Note that for all these models (actually for all KCSM introduced in physics literature), the constraints
 impose a maximal number of occupied sites in a proper neighborhood in order to allow the moves.
As a consequence the dynamics becomes slower at higher density and an
ergodicity breaking transition may occur at a finite critical density,
$p_c<1$. This threshold, as it has been formalized in Section 2.3 of \cite{Cancrini:2008p340}, corresponds to the lowest density at which the origin belongs with finite probability to a cluster of particles which are mutually and forever blocked due to the constraints.
Among the above models the North-East is the only one displaying such a
transition at $p_c<1$ (see \cite{Kordzakhia:2006p784} and \cite{Cancrini:2008p340}).  Another key feature of
KCSM is the existence of blocked configurations, namely
configurations with all creation/destruction rates identically equal to
zero. This implies the existence of several invariant measures and the
occurrence of unusually long mixing times compared to high-temperature
Ising models (see Section 7.1 of \cite{Cancrini:2008p340}).  Furthermore the
constrained dynamics is usually not attractive so that monotonicity
arguments valid for e.g. ferromagnetic stochastic Ising models cannot be
applied.
Due to the above properties the basic issues concerning the large time
behavior of the process are non-trivial.  The first rigorous results
were derived in \cite{Aldous:2002p1074} for the East model by proving that the
spectral gap of its generator is positive for all $p<1$ and also that it
shrinks faster than any polynomial in $(1-p)$ as $p\uparrow 1$.  In
\cite{Cancrini:2008p340} positivity of the spectral gap inside the ergodicity region (i.e. for $p<p_c$) has been proved
in much greater generality 
and in
particular for the N-E and FA-jf
models. This  result rejected previous conjectures
       \cite{Graham:1997p2349}    \cite{HARROWELL:1993p2347}        of a stretched
exponential decay for the time correlation function of the occupation
variable for some FA-jf models.

A key issue both from the mathematical and the physical point of view is
what happens when evolution does not start from the equilibrium measure
$\mu$. The analysis of this setting usually requires much more detailed
informations than just the positivity of the spectral gap,
e.g. positivity of the logarithmic Sobolev constant or of the entropy constant. Since this positivity does not hold for KCSM (see Section 7.1 of \cite{Cancrini:2008p340})  even the basic question of whether
convergence to $\mu$ occurs remains open in this non equilibrium setting. Of course, due to the
existence of blocked configurations, convergence to $\mu$ cannot be true
uniformly on the initial configuration and one could try to prove it a.e. or
in mean w.r.t. an initial distribution $\nu\neq \mu$.  From the point of
view of physicists a particularly relevant case (see
e.g.         \cite{Mayer:2007p2298}  \cite{Leonard:2007p1354}), is when $\nu$ is a product Bernoulli($p'$) measure
with $p'\neq p$. 
If $p'>p_c$, due to the occurrence of the forever blocked clusters,
one cannot hope to prove any such convergence result even if the starting point is chosen at random with distribution $\nu$ 
(see  Section \ref{Tree} Theorems
\ref{nonconv1} and \ref{nonconv2}).
If instead both $p$ and $p'$ are below $p_c$ the most natural guess is
that convergence to equilibrium occurs for any local function $f$ \ie
\begin{equation}
  \label{eq:basic}
\lim_{t\to \infty}\int d\nu(\h) \bbE_\h\bigl(f(\h_t)\bigr)=\mu(f)  
\end{equation}
where $\h_t$ denotes the process started from $\h$ at time $t$ and that
the limit is attained exponentially fast if the spectral gap is positive.

We start by proving \eqref{eq:basic} plus
exponential convergence for \emph{any} 
one dimensional ($\cG=\bbZ$) reversible stochastic spin model with finite range jump
rates and positive spectral gap  
when the initial distribution $\nu$ is ``not too far'' from
the reversible one (see theorem \ref{main 1}). This result covers in particular any ergodic finite range
KCSM on $\bbZ$ like the East and the FA-1f models. 

Then we turn to the East model and prove
exponential convergence to equilibrium in two cases:
\begin{enumerate}[(i)]
\item starting from a fixed
configuration $\eta$ for which no site is blocked forever, i.e. a configuration without a rightmost zero
(Theorem \ref{main 2});
\item starting from a Bernoulli ($p'$) measure for
all $p'\in[0,1)$ (Theorem \ref{main 3}). 
\end{enumerate}
These results rely heavily on
the positivity of the spectral gap and on the fact that the
East constraints are oriented (or directed), i.e. for any vertex $x$ the sites that enter in its constraint (here $x+1$)
evolve independently from the occupation variable at $x$ (see Section \ref{sec:East} for a more precise definition).  

Then we consider a KCSM which was introduced in \cite{Aldous:2002p1074} and which
we call the AD model: it is defined on a infinite rooted binary tree,
$\cG=\cT$, and the constraint requires the two children of $x$ to be vacant
in order to allow an updating of the spin at $x$. In this case $p_c=1/2$ and the
positivity of the spectral gap below $p_c$ remains to be established.  
Thanks to the fact that this is again an oriented
model and with the assumption $\gap>0$ for any $p<p_c$, we prove
exponential convergence to equilibrium when $p<p_c$ 
in two cases:
\begin{enumerate}[(i)]
\item starting from a fixed
configuration $\eta$ without an infinite cluster of $1$'s 
(see Theorem \ref{main 4});
\item starting from a Bernoulli ($p'$) measure with $p'<p_c$ (see Theorem \ref{main 5}). 
\end{enumerate}
The
proof of Theorem \ref{main 4} and \ref{main 5} clarifies the key
ingredients which are needed for our technique to work and makes it
clear why this is not the case neither for the (non oriented) FA-jf
models nor for the (oriented) North-East model.  

Finally we present in
the appendix an argument (Theorem \ref{upperEast}) which establishes the
sharp upper bound for the spectral gap of East. This result had already
be obtained in \cite{Aldous:2002p1074} via a properly devised test function. We
follow here a different and simpler strategy based on precise
``energy/entropy'' considerations.

\section{One-dimensional models near equilibrium}
\label{sec:one}
In this section we consider general KCSM on $\bbZ$. Each model is
characterized by its infinitesimal Markov generator $\cL$ whose action
on local (\ie depending on finitely many variables) functions $f:
\O\mapsto \bbR$, $\O=\{0,1\}^{\bbZ}$, is given by
\begin{equation}
\label{thegenerator}
\cL f(\omega)=\sum_{x\in \bbZ}c_{x}(\o)\left[\mu_x(f)-f(\o)\right]
\end{equation}
where $\mu_x(f)\equiv\int d\mu_x(\o_x) f(\o)$ is a function of all the
$\{\o_y\}_{y\neq x}$ and corresponds to the local mean w.r.t. to the
variable $\o_x$ computed with the Bernoulli(p) measure $\mu_x$ while all
the other variables are held fixed. 

On the coefficients (or constraints)
$c_x(\o)$ we only assume that they are the indicator function of some
non-empty event $A_x$ which depends only on the random variables $\{\o_y:\ y\neq
x, |x-y|\le r\}$ where $r\ge 1$ is some preassigned constant (finite
range condition). Since
$c_x(\o)$ does not depend on $\o_x$ one immediately checks that the product
measure $\mu=\otimes_{x\in \bbZ}\mu_x$ is a reversible stationary
measure \ie
$\cL$ can be extended to a selfadjoint, non-positive  operator on
$L^2(\O,\mu)$ which we
still denote by $\cL$. As usual we denote by $\gap(\cL)$ the spectral
gap of $\cL$ (see e.g.\cite{Gine:1997p1968}).

Finally, for any local function $f$ we will denote by $\bbE(f(\eta_t))$ or
equivalently by $P_t f(\eta)$ the expectation over the process generated
by $\cL$ at time $t$ when the
initial configuration is $\eta$.

\begin{Theorem}
  \label{main 1}
Assume $\gap(\cL)>0$. Then there exists $\l,m>0$ such that, for any
probability measure $\nu$ on $\O$ satisfying 
\begin{equation}
  \label{eq:4}
  \sup_\ell\max_{ \o_{-\ell},\dots,\o_{\ell}}\nep{-\l \ell}\frac{\nu(\o_{-\ell},\dots,\o_{\ell})}{\mu(\o_{-\ell},\dots,\o_{\ell})}<\infty
\end{equation}
and for any local function $f$ there exists $C_f<\infty$ s.t.
$$\int d\nu(\eta)\Big|\bbE(f(\eta_t))-\mu(f))\Big|\leq C_f e^{-mt}.$$
\end{Theorem}
\begin{remark}
The condition $\gap(\cL)>0$ has been verified in various popular 
one dimensional KCSM like FA-1f and East models (see \cite{Cancrini:2008p340}).
\end{remark}
Let $\L_{\ell}:=\{i\in \bbZ:\ |i|\le \ell\}$, $f$ be a local function
with support $S_f$, $\mu(f)=0$ and $\ell_f:=\inf
_{\ell\in\bbZ^+}\{\ell:S_f\subset\L_{\ell}\}$.  Fix $M>0$ to
be chosen later depending on the range $r$ and
denote by $n(t)\in\bbN$ the integer part of $Mt$. We
define the volume $\hat\L:=\L_{\ell_f+n(t)}\supset\L_{\ell_f}\supset S_f$
and by $\hat \nu$ and $\hat\mu$ the
marginal of $\nu$ and $\mu$ to $\{0,1\}^{\hat\L}$, respectively. We
also let $\{\hat \h_s\}_{s\le t}$ be the process 
up to time $t$ started from the configuration indentically equal to $0$
outside $\hat\L$ and equal to $\eta$ inside $\hat\L$. 

Using standard results of ``finite speed of propagation'' (see e.g.\cite{Martinelli:1999p274}
 ) one gets immediately that it is possible to choose $M=M(r)$ so
large that 
\begin{equation}
  \label{eq:5}
 \sup_\h |\bbE_\h\bigl(f(\h_t)\bigr)-\bbE_\h\bigl(f(\hat\h_t)\bigr)|\le C_f\nep{-t} 
\end{equation}
where $C_f$ is some constant depending on $f$. Therefore
\begin{eqnarray}
\int d\nu(\eta)\big|\bbE(f(\eta_t))-\mu(f))\big|&\le&
C_f\nep{-t}+
\int d\hat\nu(\eta)\big|\bbE(f(\hat\eta_t))-\mu(f))\big|\leq\nonumber\\
&\le & C_f\nep{-t}+ \ninf{\frac{\hat\nu}{\hat\mu}}\int d\hat\mu(\eta)\big|\bbE(f(\hat\eta_t))-\mu(f)\big|\nonumber\\
&\le & 2C_f\nep{-t}+ \ninf{\frac{\hat\nu}{\hat\mu}}\int
d\mu(\eta)\big|\bbE(f(\eta_t))-\mu(f)\big|\nonumber\\
&\le & 2C_f\nep{-t}+ \nep{2\l(\ell_f+n)}\nep{-\gap t}\Var_\mu(f)^{1/2}
\label{ine3}
\end{eqnarray}
where in the last line we used \eqref{eq:4} together with the standard
inequality $\Var_\mu(P_tf)\le \nep{-2\gap t}\Var_\mu(f)$.
If we now set $\l=\frac 14 \gap/M(r)$ we get immediately that
the r.h.s. of \eqref{ine3} is bounded from above by $C_f'\nep{-mt}$ for
some $C_f',m>0$.
\qed

\section{Non equilibrium results for the East model}
\label{sec:East}
The East model is defined on $\bbZ$ and its infinitesimal generator
$\cL$ takes the form \eqref{thegenerator} with constraints
\begin{equation}
  \label{eq:constr}
c_x(\o)=1-\o_{x+1}  
\end{equation} 
In this case and thanks to the special form of the
constraints we are able to improve considerably over Theorem \ref{main
  1} and get an optimal result. In this section $\cL$ will always denote
the generator \eqref{thegenerator} with the above special form of the constraints.
\begin{Theorem}
  \label{main 2}
  Let $\eta$ be any
  configuration s.t. there is an infinite number of $0$'s to the right
  of the origin. Then
  there exists $m>0$ and for any local function $f$ there exists
  $C_{f}<\infty$ and $t_0(\h,f)$ such that for any $t>t_0$
\begin{equation}
  \label{eq:3}
  |\bbE\bigl(f(\h_t)\bigr)-\mu(f) | \le C_f\nep{-m t}\,.
\end{equation}
\end{Theorem}

\begin{Theorem}
\label{main 3}
Fix $p'\in (0,1)$ and let $\nu$ be a Bernoulli($p'$) product measure on
$\O$. There exists $m>0$ and for any local function $f$ there exists
$C_f<\infty$ such that:
  \begin{enumerate}[{\bf a)}]
 \item  for any $t>0$
\begin{equation}
 \label{eq:main2.1}
\int d\nu(\h)\left|\bbE\bigl(f(\h_t)\bigr)-\mu(f) \right| \le C_f\nep{-m t}\,;
\end{equation}
\item for $\nu$-almost all configurations $\h$ there exists $t_0(\h,f)$
  such that for any $t>t_0(\h,f)$
\begin{equation}
  \label{eq:main2.2}
|\bbE\bigl(f(\h_t)\bigr)-\mu(f) | \le C_f\nep{-m t}\,.
\end{equation}
  \end{enumerate}
\end{Theorem}



\begin{remark}
  Theorems \ref{main 2} and Theorem \ref{main 3} can be
  extended to the version of East model on any infinite rooted tree with
  bounded connectivity analized in \cite{proceedings}.
\end{remark}
Before starting the proof of the above results it is useful to recall
that an explicit construction (sometimes refered to as the
\emph{graphical construction}) of the process
generated by $\cL$ goes as follows. Choose $p\in[0,1]$ and let
$\Bigl(\cO,\cF,\bbP\Bigr)$ be a probability space on which are defined
countably many independent rate-one Poisson processes and countably many
independent Bernoulli($p$) random variables. Assign one Poisson process
to each site $x\in \bbZ$ and one Bernoulli variable to each occurrence
of each Poisson process.  We denote by $\{t_{x,n}\}_{n\in \bbN}$ the
occurrences of the Poisson clock at site $x$ and by $\{s_{x,n}\}_{n\in
  \bbN}$ the corresponding coin tosses. The variables $\{t_{x,n}\}_{n\in
  \bbN}$ mark the possibilities for site $x$ to change its state. At
each time $t_{x,n}$ the site $x$ queries the state of its constraint
$c_x$. If it is satisfied, \ie if the spin at $x+1$ is $0$, then $x$
resets its value to the value of the corresponding Bernoulli variable
$s_{x,n}$ (see e.g. \cite{Kordzakhia:2006p784}). For notation convenience, any
occurrence of the Poisson processes such that the constraint at the site
of occurrence is satisfied will be called \emph{a legal ring}. The
process obtained in this way up to time $t$ and started from the initial
configuration $\h$ will be denoted by $\{\eta_s\}_{s\leq t}$. We stress
that the rings and coin tosses at $x$ for $s\leq t$ have no influence
whatsoever on the evolution of the configuration at the sites 
which enter in its constraint (here $x+1$) thus
they have no influence of whether a ring at $x$ for $s>t$ is legal or not. Any model sharing this property will be called {\sl oriented}. 

The next step is to recall the notion of \emph{distinguished zero}
introduced in \cite{Aldous:2002p1074}. This definition and the property stated in
Lemma \ref{distinguished} below depend crucially on the oriented nature
of the East constraints. This will be further clarified when proving a
similar result (Lemma \ref{keylemma}) for the AD model in section
\ref{Tree}. 
\begin{definition}
  Given a configuration $\h$, suppose that $\h(x)=0$ and call the site
  $x$ \emph{distinguished}. At a later time $s>0$ the position $\xi_s$
  of the distinguished zero obeys the following iterative
  rule. $\xi_s=x$ for all times $s$ strictly smaller than the first \emph{legal} ring of the mean one
  Poisson clock at $i$ when it jumps to $x+1$. Then it waits for the
  next legal ring at $x+1$ and when this occurs it jumps to $x+2$ and so
  on.
\end{definition}    
Thus, with probability one, the path $\{\xi_s\}_{s\le t}$ is
right-continuous, piecewise
constant, non decreasing, with at most a finite number of
discontinuities  at which it increases by one.  In the sequel we will
adopt the standard notation $\xi_{s^{-}}:=\lim_{\epsilon\uparrow
  0}\xi_{s+\epsilon}$. 
By exploiting the fact that the motion of the distinguished zero
for $s>t$ cannot be influenced by the clock rings and coin tosses in
$(x,\xi_{t})$, Aldous and Diaconis established the following important
result. In what follows, for any $V\sset \bbZ$ and any $\h\in \O$, we
will write $\mu_V,\, \h_V$ for
the marginal of $\mu$ on $\{0,1\}^V$ and for the restriction of $\h$ to
$V$ respectively.
\begin{Lemma}[Lemma 4 of \cite{Aldous:2002p1074})]
\label{distinguished}
Fix an interval $V_0=[x_-,x_0)$. Suppose that $\eta(x_0)=0$ and that
$\{\eta_x\}_{x=x_-}^{x_0-1}$ are distributed according to $\mu_{V_0}$.  Make $x_0$
distinguished. Then the conditional distribution of the restriction of
$\h_t$ to the set $V_t=[x_-,\xi_t)$ given the path $\{\xi_s\}_{s\le t}$ is
$\mu_{V_t}$.
\end{Lemma}
\begin{remark} Actually Aldous and Diaconis proved the above statement
  for the conditional distribution given only $\xi_t$ and not the whole
  path $\{\xi_s\}_{s\le t}$. However, as it is easily checked, the same
  proof applies in our setting.
\end{remark} 
\begin{remark}
  The main motivation behind the notion of the ``distinguished zero'' is
  the following. Given the path $\{\xi_s\}_{s\le t}$, for any pair
  $(s,y)$ satisfying $s\le t$ and $y<\xi_{s}$, the variable
  $\{\h_s(y)\}$ is uniquely determined by the occurences of the Poisson
  processes and coin tosses $\{t_{z,n},s_{z,n}\}_{n\ge 1}$ such that
  $t_{z,n}\le t$ and $z<\xi_{t_{z,n}}$ according to the following
  ``conditional graphical construction''. Without loss of generality we assume $y<x$. Until the
  first time (if it exists) the distinguished path $\{\xi_s\}_{s\le t}$ jumps from $x$ to
  $x+1$ the variables $\h$'s in the interval $[y,x-1]$ evolve according to the
  graphical construction of the usual
  East model with a fixed zero at $x$. When the path
  jumps to $x+1$ (so that all the other variables stay fixed) a new Bernoulli($p$) variable is added at the site $x$
  and the process starts again in the interval $[y,x]$.    
\label{keyremark}
\end{remark}

\begin{proof}[Proof of Theorem \ref{main 2}]
  Let for simplicity $\mu(f)=0$ and fix and interval $[x_-,x_+]$
  s.t. $S_f\subset[x_-,x_+]$. Let $x_0(\eta)$ be the position of the
  first zero to the right of $x_+$ in $\eta$ and make $x_0$
  distinguished, namely $\xi_0=x_0$ and $\xi_s$ is the position of the
  corresponding distinguished zero at time $s\le t$. Given the path
  $\{\xi_s\}_{s\le t}$, let 
$0<t_1<t_2\dots<t_{n-1}<t$ be its discontinuity points and $t_n=t$. 
We denote by
  $\{P^{(0)}_s\}_{s\le t_1}$ the Markov semigroup associated to the East
  model in the interval $V_0:=[x_-,x_0)$ with a fixed zero boundary
  condition at $x_0$. In other words we replace $\bbZ$ with $V_0$ and
  $\cL$ with the finite dimensional generator $\cL^{(0)}$ given by  
  \begin{equation}
    \label{eq:findim gen}
\cL^{(0)}f(\o)=\sum_{x\in
      V_0}c^{(0)}_{x}(\o)\left[\mu_x(f)-f(\o)\right]
  \end{equation}
where 
$$
c^{(0)}_x(\o)=
\begin{cases}
 1- \o_{x+1} &\text{ if $x\neq x_0-1$}\\
1 & \text{otherwise}
\end{cases}
$$
For any configuration $\s\in \{0,1\}^{V_0}$ we
  write $\s\otimes \s'$ for the configuration in $\{0,1\}^{[x_-,x_0]}$
  obtained from $\s$ by adding the variable $\s'\in \{0,1\}$ at the site
  $x_0$.  With these notation and thanks to the fact that the time
  evolution inside $[x_-,x_+]$ does not depend on the initial variables $\{\eta(y)\}_{y<x_-}$, we can write
\begin{gather}
  \bbE\bigl(f(\h_t)\tc \{\xi_s\}_{s\le t}\bigr)\nonumber \\
  = \sum_{\s'\in \{0,1\}}\sum_{\s\in
    \{0,1\}^{V_0}}P^{(0)}_{t_1}(\h_{V_0},\s)\mu_{x_0}(\s')\bbE\Bigl(f\bigl((\s\otimes\s')_{t-t_1}\bigr)\tc
  \{\xi_s\}_{t_1\le s\le t}\Bigr)
  \label{eq:19}
\end{gather}
where, with a slight abuse of notation, $(\s\otimes\s')_{t-t_1}$ denotes
the configuration in the interval $[x_-,\xi_t)$ obtained from the
configuration  $\s\otimes\s'$ in the interval $[x_-,x_0]$ according to the conditional graphical
construction described in Remark
\ref{keyremark} applied to the time interval $(t_1,t]$.   Therefore, if
we let $V_1:=[x_-,x_0+1)$, we get
\begin{gather}
  \Var_{\mu}\Bigl(\bbE\left(f(\h_t)\tc \{\xi_s\}_{s\le
      t}\right)\Bigr)\nonumber \\ \le \nep{-2\gap
    t_1}\Var_{\mu_{V_0}}\Bigl[\sum_{\s'\in
    \{0,1\}}\mu_{x_0}(\s')\bbE\Bigl(f\bigl((\s\otimes\s')_{t-t_1}\bigr)\tc
  \{\xi_s\}_{t_1\le s\le t}\Bigr)
  \Bigr] \nonumber \\
  \le \nep{-2\gap
    t_1}\Var_{\mu_{V_1}}\Bigl[\bbE\Bigl(f\bigl((\s\otimes\s')_{t-t_1}\bigr)\tc
  \{\xi_s\}_{t_1\le s\le t}\Bigr) \Bigr]
\label{eq:20}
\end{gather}
where $\gap>0$ is the infinite volume spectral gap for East and in the first inequality we used the fact that spectral gap of
$\cL^{(0)}$ is always
greater or equal than $\gap$ (see Lemma 2.11
of \cite{Cancrini:2008p340})
and in the second inequality we used convexity of the variance.  If
$t_1\neq t$ we can now iterate the above inequality \eqref{eq:19} for
the term inside the variance by replacing $t_1$ with the second
discontinuity point $t_2$ for the path $\{\xi_s\}_{\s\le t}$ (or by
$t_2=t$ if $\xi_t=x_1$) and by replacing $P_s^{(0)}$ with
$\{P_s^{(1)}\}_{t_1<s\leq t_2}$ defined as the Markov semigroup
associated to the East model in $V_1$ with empty boundary condition on
$x_1$. By continuing the iteration up to $t_n=t$ we get
\begin{equation}
  \label{eq:21}
  \Var_{\mu}\Bigl(\bbE\left(f(\h_t)\tc \{\xi_s\}_{s\le t}\right)\Bigr)\le \nep{-2\gap t} \Var_\mu(f).
\end{equation}
Furthermore by using $\xi_t\ge x_0>x_+$, Lemma \ref{distinguished}, the
assumption $\mu(f)=0$ and the above equality \eqref{eq:19} it follows
that $\bbE\left(f(\h_t)\tc \{\xi_s\}_{s\le t}\right)$ has $\mu$-mean
zero w.r.t. the initial configuration $\eta$, namely
\begin{equation}
\label{meanzero}
\int d\mu(\eta)\bbE\left(f(\h_t)\tc \{\xi_s\}_{s\le t}\right)=0
\end{equation}
Finally, putting together \eqref{eq:21} and \eqref{meanzero} yields
\begin{gather}
    \Big|\,\bbE\Bigl(f(\h_t)\Bigr)\Big|\le \bbE\Big(\big| \bbE\bigl(f(\h_t)\tc \{\xi_s\}_{s\le t}\bigr)\big|\Bigr)\nonumber \\
\le \Bigl(1/(p\wedge q)\Bigr)^{x_0-x_-}\bbE\Big(\int d\mu(\h)\big|\,\bbE\bigl(f(\h_t)\tc \{\xi_s\}_{s\le t}\bigr)\big|\Bigr) \nonumber \\
\le \Bigl(1/(p\wedge q)\Bigr)^{x_0-x_-} \bbE\Big(\Var_{\mu}\Bigl(\bbE\left(f(\h_t)\tc \{\xi_s\}_{s\le t}\right)\Bigr)^{1/2}\Bigr)
\nonumber \\
\le \Bigl(1/(p\wedge q)\Bigr)^{x_0-x_-}\nep{-\gap t} \Var_\mu(f)^{1/2}
\label{eq:23}
\end{gather}
where to obtain the third inequality we used Cauchy-Schwartz inequality
and \eqref{meanzero}.  The claim is proved by taking
$C_{f}=\left(1/(p\wedge q)\right)^{x_+-x_-}\Var_\mu(f)^{1/2}$,
$~m=\gap/2$ and $t_0(f,\eta)=2(x_0(\eta)-x_+)|\log\left(p\wedge
  q\right)|1/\gap$.
\end{proof}

\begin{proof}[Proof of Theorem \ref{main 3}]
  Part (b) follows immediately from theorem \ref{main 2}.
  In order to prove part
  (a) we use the same notation as above and, for a given $\d>0$ and
  local $f$,
  we let $\cA_{\delta,t}:=\{\h:\
  x_0(\h)-x_+(f)\ge \d t\}$.  Clearly $\nu(\cA_{\delta,t})=(p')^{\delta
    t}$. We can then split the average $\int
  d\nu(\h)\left|\bbE\bigl(f(\h_t)\bigr)-\mu(f)\right|$ into the contribution from
  $\eta\in \cA_{\delta,t}$ and $\eta\not\in\cA_{\delta,t}$. By choosing $\d=\gap/(2|\log\bigl(p\wedge q\bigr)|)$ and $C_f$ as above, we
  immediately get
\begin{equation}
  \label{eq:26}
  \int d\nu(\h)\big|\bbE\Bigl(f(\h_t)\Bigr)-\mu(f)\big| \le 
\ninf{f}\nep{-c_\d t} +  C_f \nep{-\frac 12 \gap t}
\end{equation}
with $c_{\delta}>0$ since $p'\neq 1$.
\end{proof}

\section{Non equilibrium results for the AD model}
\label{Tree}
As already mentioned in the introduction the AD model is defined on the (infinite)
binary rooted tree $\cT$ with root $0$. Its generator takes the form \eqref{thegenerator} with
constraints given by 
\begin{equation}
  \label{eq:AD1}
  c_x(\o)=
  \begin{cases}
    1 & \text{if both children of $x$ are zero}\\
0 & \text{otherwise}.
  \end{cases}
\end{equation}
In this section $\cL$ will always denote
the generator \eqref{thegenerator} with the above special form of the constraints. Note that, as for East, this choice is oriented: if we
make the graphical construction as in Section \ref{sec:East} it is immediate to verify that the rings and coin tosses at $x$ for $s\leq t$ have no influence in the evolution of its two children, thus they do not
influence the fact that a ring at $x$ for $s>t$ is legal or not.
In order to state our results we need to introduce some
notation of site percolation on the
tree. 
We call {\sl path} any sequence $\{x^{0},x^{1},\dots x^n\}$ of distinct
points in $\cT$ such that, for all $i$, $x^{i}$ is the parent of
$x^{(i+1)}$.  For a given configuration $\eta$ we say that $x\rightarrow
y$ if there is a path of occupied sites starting in $x$ and ending in
$y$ (thus $x\rightarrow x$ iff $\eta(x)=1$). We also define the {\sl
  occupied cluster} of $x$ as the random set
$$
\cC_x(\eta):=\{y\in {\mathcal{T}}:x\rightarrow y\}.
$$ 
Let
$\cP_x^{\ell}:=\{\eta:|\cC_x(\eta)|\geq\ell\}$, let
$\cP^\infty_x:=\{\eta:|\cC_x(\eta)|=\infty\}$ and let
$\theta(p):=\mu(\cP^{\infty}_{0})$ (recall that $\mu$ is the
Bernoulli($p$) product measure on $\{0,1\}^{\cT}$). The corresponding
site percolation critical density is defined as
$$
p_{sp}:=\sup\{p\in[0,1]:\theta(p)=0\}
$$
and, thanks to Proposition 2.5 of \cite{Cancrini:2008p340}, it coincides with the treshold of the ergodicity regime for AD model, namely 
\begin{equation}\label{equiva}
p_{sp}=p_c
\end{equation} with
$$
p_c:=\sup\{p\in[0,1]:0 {\mbox{ is simple eigenvalue of }} \cL\}.
$$
The following results are well known (see for example   \cite{Grimmett:1999p2693}  )
\begin{Proposition}\
\begin{enumerate}[i)]
\item $p_{sp}=1/2$
\item If $p<1/2$ there exists $\beta(p)>0$
such that
\begin{equation}
\label{percoexpdecay}
\lim_{n\to\infty}\frac{1}{n}|\log \mu(\id_{\cP_0^{n}})|\geq\beta
\end{equation}
\item If $p=1/2$ there exists $c_1,c_2>0$ s.t.
\begin{equation}
\label{powerdecay}
\frac{c_1}{n}<\mu(\id_{\cP_0^{n}})<\frac{c_2}{n}
\end{equation}
\end{enumerate}
\label{propoperco}
\end{Proposition}
As a consequence
of the existence of an infinite percolation cluster above $p_c$ it is
very easy to see (simply use the test
function $f(\h)= \id_{\cP^\infty_0}(\h)$),  that $\gap(\cL)=0$ for
$p>p_c$. The same holds at the critical case $p=p_c$ with a slightly
subtler proof. Completely open is instead the interesting conjecture 
made in \cite{Aldous:2002p1074} that $\gap(\cL) >0$ for $p<p_c$. \emph{In all what
follows we will always assume that this is the case}.  

We are now ready to
state our results. In what follows $\nu$ will always denotes the
Bernoulli($p'$) product measure on $\cT$. 
\begin{Theorem}
  Let $\eta$ be a configuration s.t. $|\cC_x(\eta)|<\infty$ for all $x$
  and assume that $\gap(\cL)>0$.  Then the same exponential convergence
  result as in Theorem \ref{main 2} hold true . 
\label{main 4}
\end{Theorem}
\begin{Theorem}
Let $0\leq p'<p_c$ and  assume that $\gap(\cL)>0$ (so that necessarily $p<p_c$). Then
the same exponential convergence results as in Theorem \ref{main 3} hold
true.
\label{main 5}
\end{Theorem}
We shall now explore the regime outside the validity of the hypothesis
for Theorem \ref{main 5}. 
\begin{Theorem}
\label{nonconv1}
If $p\leq p_c<p'$ then for any $c>0$ there exists a local $f$ s.t. for all $t>0$
$$
\left|\int d\nu(\eta)\bbE(f(\eta_t)-\mu(f))\right|\geq c.
$$
\end{Theorem}
If instead $p_c=p'$ and $p<p_c$ we cannot exclude convergence to equilibrium
but we can set a bound on the speed of convergence which excludes exponential convergence
\begin{Theorem}
\label{nonconv2}
If $p< p_c=p'$ then for any $c>0$ there exists a local $f$ s.t. for all $t>0$
$$\int d\nu(\eta)\left|\bbE(f(\eta_t)-\mu(f))\right|\geq \frac{c}{t^2}.$$
\end{Theorem}
The regime $p'<p_c\leq p$ remains to be explored. We conjecture that at least for sufficiently high $p$ there exist local functions that do not converge to equilibrium. This conjecture is corroborated by the fact that we can prove  this result  for the following modified AD model.

Consider  a non-rooted tree $\widetilde \cT$ with connectivity three and let the constraint require at least two empty nearest neighbours. On this graph we can define as before the occupied clusters and the site percolation critical density which again coincides with the ergodicity threshold, $p_c$.
Then we recall that on $\widetilde\cT$
 Theorem 1.6 of
\cite{Haggstrom:1997p2123} proves that if the local density is sufficiently large there exists necessarily an infinite percolation cluster. More precisely
if $\tilde\nu$ is an automorphism invariant measure on $\widetilde \cT$, then there
exists $1/2<\delta<1$ such that
 $\tilde\nu(\eta(0)\eta(1))\geq\delta$ implies $\tilde\nu(\cP^\infty_0)>0$.
Let $\nu_t$ denote the evoluted of time $t$ of the initial Bernoulli(p') measure $\nu$ with $p'<p_c$. 
Thanks to the translation invariance of the constraints and
  of the initial measure, $\nu_t$ is also translation
  invariant. Furthermore, since the characteristic function $\id_{\cP_0}$ is
  left invariant by the dynamics (an infinite cluster can neither be
  created nor disrupted), it holds
  $\nu_t(\id_{\cP_0})=\nu(\id_{\cP_0})=\theta(p')=0$, where the latter
  equality follows from $p'<p_c$. Therefore if we set
 $f=\eta(0)\eta(1)$ we have necessaily $\int
  d\nu(\eta)\bbE(f(\eta_t))=\nu_t(f)<\delta$, otherwise $\nu_t(\id_{\cP_0})=0$
 would be in contradiction with Haggstrom theorem. This inequality together
  with $\mu(f)=p^2$ yields for any $t>0$ 
$$|\int d\nu(\eta)\bbE(f(\eta_t))-\mu(f)|>p^2-\delta>0.$$ Thus for this modified (and non oriented) model we have identified a local function which does not converge to equilibrium in the regime $p'<p_c$ and $p>\sqrt\delta$.

\subsection{The distinguished set of zeros and its properties}
In analogy with the definition of ``distinguished zero'' introduced in
the analysis of the East model we will begin by defining the
\emph{distinguished set of zeros}. In what follows, for any $x\in \cT$,  $\cT_x$ will denote
the binary tree rooted at $x$ and $\cK_x$ the two children of $x$. For a
subset $\L\sset \cT$ the set of vertices outside $\L$ but such that
their parent is in $\L$ will be denoted by $\partial_+\L$.
\begin{definition}
\label{defini}
Consider a  region $\Lambda\subset{\mathcal{T}}$ with the property that 
\begin{equation}
\label{keyprop}
\left(\cup_{x\in\partial_+\L}{\mathcal{T}}_{x}\right)\cap \L=\emptyset
\end{equation}
and a configuration $\eta$ such
that  
\begin{equation}
\label{keyprop2} 
\eta(x)=0~~~~ \forall x\in\partial_+\Lambda
\end{equation}
We define the \emph{distinguished set of
  zeros} $B_{t=0}$ and the
\emph{distinguished volume} $V_{t=0}$ at time $t=0$ to be the sets
$\partial_+\Lambda$ and $\L$ respectively. At
a later time $s>0$, the distinguished set of zeros and the distinguished
volume are defined as
follows.  $V_s=V_0$ and $B_s=B_0$ until the first time $t_1$ at which a
legal ring occurs for one of the Poisson clocks at the sites in
$B_0$. Call $x_0$ this site. Then we set $V_{t_1}=V_0\cup x_0$ and
$B_{t_1}=\Bigl(B_0\cup \cK_{x_0}\Bigr)\setminus \{x_0\}=\partial_+ V_{t_1}$ and the rule is iteratively applied to
define the distinguished volume and border at any later time.
\end{definition}

Note that, for any $t<\infty$, with probability one there
are at most a finite number of times $0<t_1<t_2\dots<t_n<t$ such that
$V_{t_i+1}\neq V_{t_i}$ and $B_{t_i+1}\neq B_{t_i}$.
For all $s<t$ the following properties hold:
\begin{claim}\ 
\label{claimpro}
\begin{enumerate}[i)]
\item   $V_t\supseteq V_s$;
\item $V_t\subseteq \L\cup(\cup_{x\in\partial \L_+}{\mathcal{T}}_x)$;
\item $B_t=\partial_+V_t\subseteq \cup_{x\in\partial \L_+}{\mathcal{T}}_x$;
\item
$\eta_t(x)=0$  $\forall x\in B_t$;
\item
if $x\neq y$ and   $x,y\in B_t$, then
${\mathcal{T}}_{x}\cap\cT_y=\emptyset$ ;
\item 
$\left(\cup_{x\in B_t}{\mathcal{T}}_{x}\right)\cap V_t=\emptyset$;
\item for all $i$, given $V_{t_i}$ and $t_i$, the random variable
  $t_{i+1}-t_i$ does not depend on the occurrences of the Poisson clocks
  at sites $x\in V_{t_i}$ for times $t>t_i$ neither on the corresponding
  coin tosses.
\end{enumerate}

\end{claim}
\begin{proof}
  (i), (ii), (iii) and (iv) follow immediately from definition
  \ref{defini}. 

  (v): Let $x,y\in B_0$ and assume by contradiction that
  ${\mathcal{T}}_{x}\cap{\mathcal{T}}_y\neq\emptyset$. Then, thanks to
  the tree structure, either $x\in{\mathcal{T}}_y$ (and
  $\cT_x\subset\cT_y$) or $y\in{\mathcal{T}}_x$ (and
  $\cT_y\subset\cT_x$). Consider the former case (the other may be
  treated analogously) and call $z$ the ancestor of $x$. Since $x\ne y$,
  $z$ also belongs to ${\mathcal{T}}_y$.  But this is in contradiction
  with hypothesis \eqref{keyprop}, since $x\in\partial_+\L$ implies
  $z\in\L$. Thus property (v) holds at $t=0$. Let us proceed by
  induction: suppose (v) holds up to $t_i$ (and therefore also for
  $t_i<s<t_{i+1}$), we will prove that it holds at $t_{i+1}$. Let
  $x,y\in B_{t_{i+1}}$. Since $B_{t_1+1}=(B_{t_i}\setminus x_i)\cup
  \cK_{x_i}$ either $x,y\in B_{t_i}\setminus x_i$ or $x,y\in \cK_{x_i}$
  or $x\in B_{t_i}\setminus x_i$, $y\in\cK_{x_i}$ (or the
  converse). Property (v) follows immediately in the first case by the
  induction hypothesis, in the second case by the tree structure, in the
  third case because
  $y\in\cT_{x_i}$ and by the induction hypothesis.

  (vi): At time zero the property holds by Definition \ref{defini}. Let
  us suppose it holds at $t_i$, we will now prove it holds at
  $t_{i+1}$. From Definition \ref{defini} it is immediate to verify that
$$\cup_{x\in B_{t_{i+1}}}\cT_x~\cap~ V_{t_{i+1}}=$$
$$=
\left(\cup_{x\in B_{t_{i}}\setminus x_i}\cT_x\cap V_{t_{i}}\right) \cup
\left(\cup_{x\in \cK_{x_{i}}}\cT_x\cap V_{t_{i}}\right) \cup
\left(\cup_{x\in B_{t_{i}}\setminus x_i}\cT_x\cap
  x_i\right)\cup\left(\cup_{x\in \cK_{x_{i}}}\cT_x\cap x_i\right)$$ The
proof is then completed by noticing that all the above sets are empty:
the first and second ones thanks to the recursive hypothesis ($x\in
\cK_{x_i}$ implies $\cT_x\subset \cT_{x_i}$), the third one thanks to
property (v) (note that $x_i\in B_{t_i}$ and $x_i\in\cT_{x_i}$), the
forth one because $x\in\cK_{x_i}$ implies
$x_i\not\in\cT_x$.

(vii): the time $t_{i+1}-t_i$ is the time before the first legal ring of a
site $x\in B_{t_i}$ and it clearly depends only on the Poisson clocks and coin tosses
at $\cup_{x\in B_{t_i}}\cT_x$. The desired independency property then
follows from property (vi).
\end{proof}
We are now ready to state the analog of Lemma \ref{distinguished}
\begin{Lemma}
\label{keylemma}
Consider a region $\L$ and a configuration $\eta$ which satisfy the
hypothesis \eqref{keyprop} and \eqref{keyprop2} of definition
\ref{defini} and make $\L$ and $\partial_{+}\L$ distinguished.  If the
restriction of $\eta$ to $\Lambda$ is distributed according to  $\mu_{\Lambda}$,
then for each $t>0$ the conditional distribution of $\eta_t$ restricted
to $V_t$ given $\{V_s\}_{s\leq
  t}$   is $\mu_{V_t}$.
\end{Lemma}
\begin{proof}
  Let $\tilde{\nu_t}$ be the marginal on $V_t$ of the conditional
  distribution of $\eta_t$ given
  $\{V_s\}_{s\leq t}$ and distinguish two cases.\\
  (a) $V_t=V_0$.  The evolution up to time $t$ inside $V_0$ is therefore
  the evolution of the model with empty boundary condition on
  $\partial_{+}V_0$. This is true thanks to the property (iv) of Claim
  \ref{claimpro}.  Denoting by $P^{(0)}_t$ the corresponding Markov semigroup
  on $\Omega_{V_t}=\Omega_{V_0}=\{0,1\}^{|V_0|}$ and recalling that
  $P^{(0)}_t$ is reversible with respect to $\mu_{V_0}$ we immediately get
  $\tilde\nu_t(\sigma)=\mu_{V_t}(\eta)$.\\
  (b) $V_t\neq V_0$. We denote by $t_1,t_2,\dots t_n$ the subsequent
  times $0<t_1<t_2\dots t_n<t$ at which $V_s$ changes and by
  $\eta_{t^-_i}$ ($\eta_{t^+_i}$) the configurations before (after) the
  change occurring at $t_i$. By the previous argument it is immediate to
  verify that $\tilde\nu_{t_1^{-}}=\mu_{V_{t_1^-}}$. We shall now assume
  inductively that $\tilde\nu_{t_i^{-}}=\mu_{V_{t_i^-}}$ and prove that
  $\tilde\nu_{t_{i+1}^{-}}=\mu_{V_{t_{i+1}^-}}$.  If we denote by $x_i$
  the site belonging to $B_{t_{i^-}}$ on which the legal rings occurs at
  $t_i$, namely $x_i=V_{t_{i+1}}\setminus V_{t_i}$, the restriction of
  $\eta_{t_i^+}$ to $V_{t_{i+1}}$ is given by the restriction of
  $\eta_{t_i^-}$ to $V_{t_i^{-}}$ plus an independent Bernoulli (p)
  random variable at site $x_i$.  Thus it follows immediately from the
  induction hypothesis that $\tilde\nu_{t_{i}^{+}}=\mu_{V_{t_{i}^+}}$.
  Then, noticing that $B_{t^+_i}=\partial_{+}V_{t^+_i}$ stays empty up
  to time $t_{i+1}^-$, we can denote by $P^{(i+1)}_t$ the Markov semigroup with
  empty boundary conditions on $\O_{V_{t_i}^+}$ and apply the same
  argument as in (a)
  to conclude that $\tilde\nu_{t_{i+1}^-}=\mu_{V_{t_{i+1}^-}}$.\\
\end{proof}
\subsection{Proof of the Theorems}
\begin{proof}[Proof of Theorem \ref{main 4}]
  The proof follows the same pattern of the proof of Theorem \ref{main
    1}. For a given local function $f$ with support $S_f\subset {\mathcal{T}}$
  we denote by ${\mathcal{T}}_f$ the smallest regular substree of $\cT$
  containing $S_f$. Given a configuration $\eta$ such any percolation
  cluster $\cC_x(\h)$ is finite,
we denote by $\cA_f(\eta)=\cup_{y\in \partial_+
  {\mathcal{T}}_{f}}\cC_y(\h)$.
If we set $V_0(\eta):=\cA_f(\eta)\cup {\mathcal{T}}_{f}$, $V_0$ clearly verifies
property \eqref{keyprop} and \eqref{keyprop2}.  Therefore we can make $V_0$ and $\partial_+ V_0$ distinguished at
time $0$ and call $V_t$ and $B_t$ the corresponding distinguished sets
at time $t$. Given the path $\{V_s\}_{s\leq t}$, we denote by $t_1$ be
the first time at which $V_s \neq V_0$ and by $\{P_s^{(0)}\}_{s\le t_1}$ the
Markov semigroup associated to AD model on $V_0$ with empty boundary
conditions on $\partial_{+}V_0$ (as in \eqref{eq:findim gen}).  Then we get
\begin{eqnarray}
  & \bbE(f(\eta_t)|\{V_s\}_{s\leq t})=\nonumber\\
  & =\sum_{\sigma'\in\{0,1\}}\sum_{\sigma\in\{0,1\}^{V_0}}P^0_{t_1}(\eta_{V_0},\sigma)\mu_{x_0}(\sigma')\bbE(f(({\sigma\times\sigma'})_{t-t_1})|\{V_s\}_{t_1\leq s\leq t}).\nonumber\\
  &\end{eqnarray}
In analogy with what we did to derive \eqref{eq:21} for the East model and under the hypothesis $\gap(\cL)>0$, it follows that
\begin{equation}
\Var_{\mu}\left(\bbE(f(\eta_t)|\{V_s\}_{s\leq t})\right)\leq e^{-2\gap t} \Var_{\mu}(f).
\end{equation}
Lemma \ref{keylemma} yields 
\begin{equation}
\label{meanAD}
\int d\mu(\eta) \bbE(f(\eta_t)|\{V_s\}_{s\leq t})=0.
\end{equation}
and again in analogy to the procedure used to derive \eqref{eq:23} we get
\begin{equation}
\label{mainAD}
|\bbE\left(f(\eta_t)\right)|\leq (1/(p\wedge q))^{|V_0(\eta)|}e^{-\gap t} \left(\Var_{\mu}(f)\right)^{1/2}.
\end{equation}
The proof is then completed by choosing $C_f=(1/(p\wedge q))^{|\cT^f|}\Var_{\mu}(f)^{1/2}$, $m=\gap(\cL)/2$ and $t_0(f,\eta)=-2|\cA_f(\eta)|\log(p\wedge q)|/\gap(\cL)$.

\end{proof}

\begin{proof}[Proof of Theorem \ref{main 5}]
  The proof of part (b) follows immediately from Theorem \ref{main 4}
  and $p<p_c$. Part (a): for a chosen $\delta$ and $f$ we let
  $\cE_{\delta,t}:=\{\eta:|\cA_f(\eta)|>\delta t\}$. Thanks to
  \eqref{percoexpdecay}, asymptotically in $t$, we can bound the
  probability of this event as
\begin{eqnarray}
&\nu(\cE_{\delta, t})\leq p^{-2^{\ell}}\nu(|\cC_0|>2^{\ell}+\delta t)\leq 
p^{-2^{\ell}}\exp(-\delta t\beta)
\label{event}
\end{eqnarray}
Thus, if we use \eqref{mainAD} together with \eqref{event} and choose
$C_f$ as for Theorem \ref{main 3} and $\delta=\gap /(2|\log (p\wedge
q)|$ we get
\begin{equation}
\int d\nu(\eta)\Big|\bbE(t(\eta_t))\Big|\leq ||f||p^{-2^{\ell}}e^{-tc_{\delta}}+C_fe^{-t\gap/2}
\end{equation}

\end{proof}
\begin{proof}[Proof of Theorem \ref{nonconv1}]
  Fix $\ell$ and let $f$ and $g$ be the characteristic functions of the
  event that the occupied cluster of the origin has cardinality at least
  $\ell$ and infinite cardinality, respectively. Namely
  $f=\id_{\cP_0^{\ell}}$ and $g=\id_{\cP_0}$ (thus $f$ is local and $g$
  is not local). Then for any choice of $\ell$ it holds
  $f(\eta)>g(\eta)$. Since $g$ is left invariant by the dynamics we have
$$\int d\nu(\eta)\bbE\Big(f(\eta_t)\Big)\geq\nu(\id_{\cP_0})=\theta(p')>0.$$
On the other hand from Proposition \ref{propoperco} (case (ii) if
$p<p_c$ or the upper bound of case (iii) if $p=p_c$), provided $\ell$ is
chosen sufficiently large, $\ell>\bar\ell(c,p')$, it holds
$$\mu(f)\leq \theta(p')-c$$
and the proof is concluded.
\end{proof}

\begin{proof}[Proof of Theorem \ref{nonconv2}] We start by inequality
$$\int d\nu(\eta)\int d\nu(\sigma)\left(P_tf(\eta)-P_tf(\sigma)\right)^2=\Var_{\nu}(P_tf)\leq ||f||_{\infty}\nu(|P_tf|).$$
Then we lower bound the left hand side by requiring that:(i)
$\eta(0)=1$, (ii)$|C_0(\eta)|\geq 2t$, (iii) $\sigma(0)=0$, (iv)
$|C_0(\sigma)|\geq 2t$ and (v) up to time $t$ neither for the evoluted
of $\eta$ nor for the evoluted of $\sigma$ the ordered sequence of $2t$
rings necessary to make the origin unconstrained has occurred.  By using
the lower bound in Proposition \ref{propoperco} (iii) to bound the
events (ii) and (iv) and the large deviation for the Poisson
distribution for event (v) we get the desired result.
\end{proof}

\section{Appendix: an alternative proof for the upper bound on the spectral gap of the East model}

Consider East model on $\bbZ$ and let $q:=1-p$. We will now present an
argument different from the one in Section 5 of  \cite{Aldous:2002p1074}
to prove
a sharp upper bound for its spectral gap. This will help to further clarify
the role played by dynamical energy barriers.
\begin{Theorem} There exists a constant $C$ independent of $q<1/2$ such that
$\gap<Cq^{-2}q^{\log(1/q)/(2\log 2)}$
\label{upperEast}
\end{Theorem}
\begin{proof}
  Let $\ell=1/q$. By using Lemma 2.11 of \cite{Cancrini:2008p340} we can upper bound the
  spectral gap on $\bbZ$ with the one on $[0,\ell)$ with zero boundary
  condition on $x=\ell$ (defined by \eqref{eq:findim gen} with $V_0=[0,\ell)$), which we call
  $\gap_{\ell}$.  With a little abuse of notation of
 we write here $\mu$ for $\mu_{[0,\ell)}$.  If we consider the
  East model Markov chain in $[0,\ell)$ in discrete time (at each step choose a random site in
  $[0,\ell)$ and
  try to update the current configuration with the correct East probabilities), the obvious relation
$$
\gap_{\ell}=\ell \gap_{\rm d,\ell}
$$
holds true, where $\gap_{d,\ell}$ is the spectral gap in the discrete time
setting. In the sequel we denote by $\vec 1$ the completely filled
configuration and by $\bbP_{\vec 1}(A)$, where $A$ is an event which
depends on ${\eta_s}_{\{s\geq 0\}}$, the probability of $A$ under the
discrete evolution started at time zero from $\vec 1$. Finally we denote
by $T$ the first time there are $n\equiv\inte{\log_2(\ell)}$ zeros and by $T_0$
the first time there is exactly one zero
located at the origin.

If the process starts from $\vec 1$, then $T_0\ge T$.  In fact, on the
half lattice with zero boundary condition, starting from all ones and
under the condition that at most $n$ zeros can be created, in
\cite{Chung:2001p1140} it has been proven that:
\begin{enumerate}
\item the minimum distance from 
the origin 
of the zero in a configuration
  with only one zero is 
  $\ell-2^{n-1}$;
\item the set $\O_0$ of
different configurations that the chain can
explore has cardinality
$2^{n\choose{2}}n!\,c^n$ with $c\approx 0.67$. 
\end{enumerate}
Thus up to time $T$ the cardinality of the set $\O_0$ of accessible configurations is at most
$2^{n\choose{2}}n!\,c^n$ and necessarily (provided $1/(2q)=\ell/2>1$) $T\le T_0$ since otherwise the
configuration with exactly one single zero at the origin would have
been unreachable. 
\begin{remark}
The entropic factor  $2^{n\choose{2}}n!\,c^n$ is much smaller (for small
$q$) than the binomial entropic factor $\ell \choose n$.
\end{remark}
We denote by $\O_n\sset \O_0$ those configurations with exactly $n$
zeros.  For $t=\frac{1}{2}q^{-(n/2)(1-o(1))}$ and $q$ small enough
we have:
\begin{align*}
\bbP_{\vec 1}(T<t) &\le t\sum_{\s\in \O_n}\sup_{s}\bbP_{\vec 1}(\s_s=\s) \\
&\le  t\,2^{n\choose{2}}n!\,c^n \, \frac{1}{\mu(\vec 1)}q^n p^{\ell
  -n} \\
&\le t q^{(n/2)(1-o(1))}\le \frac 12 
\end{align*}
Next we recall that if $T_A$ denotes
the hitting time of a set $A$, then
\begin{equation*}
  \bbP_\mu(T_A\ge t)\le e^{-\l_At}
\end{equation*}
where (${\cD}_{\rm d}(f)$ is the discrete-time Dirichlet form of $f$)
\begin{equation*}
  \l_A:= \inf\Bigl\{{\cD}_{\rm d}(f):\ \mu(f^2)=1,\ f\equiv 0\text{ on
    } A\Bigr\}\ge \mu(A)\gap_{\rm d,\ell}
\end{equation*}
We apply the above observation to the set $A$ consisting of the single configuration with
all ones except the origin, thus $\mu(A)=q(1-q)^{n-1}$ and $T_A=T_0$. 
For $t=\frac{1}{2}q^{-(n/2)(1-o(1))}$ we get  
\begin{gather*}
  \label{eq:2}
  e^{-t\mu(A)\gap_{\rm d,\ell}}\ge \bbP_\mu(T_0\ge t)
\ge p^{\ell}\,\bbP_1(T_0>t) 
\ge p^{\ell}\,\bbP_1(T>t)\ge \frac{e^{-1}}{2} 
\end{gather*}
which implies that 
$$
t\mu(A)\gap_{\rm d,\ell}\le 1+\log(2)\,.
$$
In conclusion 
\begin{equation}
  \gap\leq\gap_{\ell} =\ell \gap_{\rm d,\ell} \le  C q^{-2}q^{(n/2)(1-o(1))}.
\end{equation}
\end{proof}

\bibliographystyle{amsplain}
\bibliography{NEQ}

\\

\end{document}